\documentclass[reqno,11pt]{birkart}
\usepackage{latexsym}
\usepackage{amssymb}
\usepackage{amsmath, graphicx, amscd}

\theoremstyle{plain}
\newtheorem{theorem}{Theorem}		
\newtheorem*{corollary}{Corollary}
\newtheorem{proposition}[theorem]{Proposition}   
\theoremstyle{definition}
\newtheorem*{defin}{Definition}
\newtheorem{definition}[theorem]{Definition}    
\newtheorem{remark}[theorem]{Remark}		
\newtheorem*{example}{Example}

\numberwithin{equation}{section}
\numberwithin{theorem}{section}

\def\ttilde{\widetilde}

\begin{document}

\title[Clifford Algebras and Euclid's parameterization of Pythagorean triples]
{Clifford Algebras and Euclid's \\
     Parameterization of Pythagorean Triples}

\author{Jerzy Kocik}
\address{Department of Mathematics, Southern Illinois University, Carbondale, IL 62901}
\email{jkocik@math.siu.edu}

\begin{abstract} We show that the space of Euclid's parameters for 
Pythagorean triples is endowed with a natural symplectic structure and that 
it emerges as a spinor space of the Clifford algebra $\mathbb{R}_{21}$, 
whose minimal version may be conceptualized as a 4-dimensional real algebra 
of ``kwaternions.'' We observe that this makes Euclid's parameterization the 
earliest appearance of the concept of spinors. We present an analogue of the 
``magic correspondence'' for the spinor representation of Minkowski space 
and show how the Hall matrices fit into the scheme. The latter obtain an 
interesting and perhaps unexpected geometric meaning as certain symmetries 
of an Apollonian gasket. An extension to more variables is proposed and 
explicit formulae for generating all Pythagorean quadruples, hexads, 
and decuples are provided. 
\\
\\
{\bf Keywords:} Pythagorean triples, Euclid's parameterization, spinors, Clifford 
algebra, Minkowski space, pseudo-quaternions, modular group, Hall matrices, 
Apolloniam gasket, Lorentz group, Pythagorean quadruples and $n$-tuples. 
\end{abstract}


\maketitle
%
%
\section{Introduction }  \label{s:1}

In the common perception, Clifford algebras belong to modern mathematics. It 
may then be a point of surprise that the first appearance of spinors in 
mathematical literature is already more than two thousand years old! 
Euclid's parameters $m$ and $n$ of the Pythagorean triples, first described by 
Euclid in his \textit{Elements}, deserve the name of Pythagorean spinors --- as will be 
shown.

The existence of Pythagorean triples, that is triples of natural numbers 
$(a,b,c)$ satisfying
\begin{equation}    \label{eq:1.1}
  a^{2}+b^2=c^{2}
\end{equation}
has been known for thousands of years\footnote{ The cuneiform tablet known 
as Plimpton 322 from Mesopotamia enlists 15 Pythagorean triples and is 
dated for almost 2000 BCE. The second pyramid of Giza is based on the 3-4-5 
triangle quite perfectly and was build before 2500 BCE. It has also been 
argued that many megalithic constructions include Pythagorean triples \cite{Tho}.}. 
Euclid (ca 300 \textsc{bce}) provided a formula for finding 
Pythagorean triples from any two positive integers $m$ and $n$, $m>n$, namely:
\begin{equation}    \label{eq:1.2}
  \begin{aligned}
  a &= m^{2} - n^{2} \\
  b &= 2\textit{mn}  \\
  c &= m^{2}+n^{2}
\end{aligned}
\end{equation}
(Lemma 1 of Book X). It easy to check that $a$, $b$ and $c$ so defined automatically satisfy Eq. (1.1), 
i.e., they form an integer right triangle.

A Pythagorean triple is called \textit{primitive} if ($a$, $b$, $c)$ are mutually prime, that is 
$\hbox{gcd\,}(a,b,c)=1$. Every Pythagorean triple is a multiple of some primitive triples. 
The primitive triples are in one-to-one correspondence with relatively prime pairs 
$(m, n)$, $\hbox{gcd\,}(m,n)=1$, $m>n$, such that exactly one of ($m$, $n)$ is even \cite{Sie, T-T}.

An indication that the pair ($m, n)$ forms a \textit{spinor} description of Pythagorean triples 
comes from the well-known fact that \eqref{eq:1.2} may be viewed as the square of an integer 
complex number. 
After recalling this in the next section, we build a more profound analysis based 
on a 1:2 correspondence of the integer subgroups of $O(2,1)$ and $SL^{\pm\!}(2,\mathbb{R})$, 
from which the latter may be viewed as the corresponding pin group of the former. 
For that purpose we shall also introduce the concept of pseudo-quaternions 
(``kwaternions'') --- representing the minimal Clifford algebra for the pseudo-Euclidean 
space $\mathbb{R}^{2,1}$ . 

The most surprising result concerns an Apollonian gasket where all the 
objects mentioned acquire a geometric interpretation.

\section{Euclid's Labels as Complex Numbers }  \label{s:2}

Squaring an integer complex number $z=m$+\textit{ni} will result in an integer complex number
\begin{equation}    \label{eq:2.1}
    z^{2}=(m+n\,i)^2=(m^2-n^2)+2mn\;i,
\end{equation}
the norm of which turns out to be also an integer:
\begin{equation}    \label{eq:2.2}
  \vert z^2\,\vert \;\;=\;\;m^2+n^2\,.
\end{equation}
This auspicious property gives a method of producing Pythagorean triangles: 
just square any integer complex number and draw the result. 
For instance $z = 2 + i$ will produce 3-4-5 triangle, the so-called Egyptian triangle 
(Figure~\ref{fig:2}b). 
\begin{figure}  
\[
  \includegraphics[width=5in]{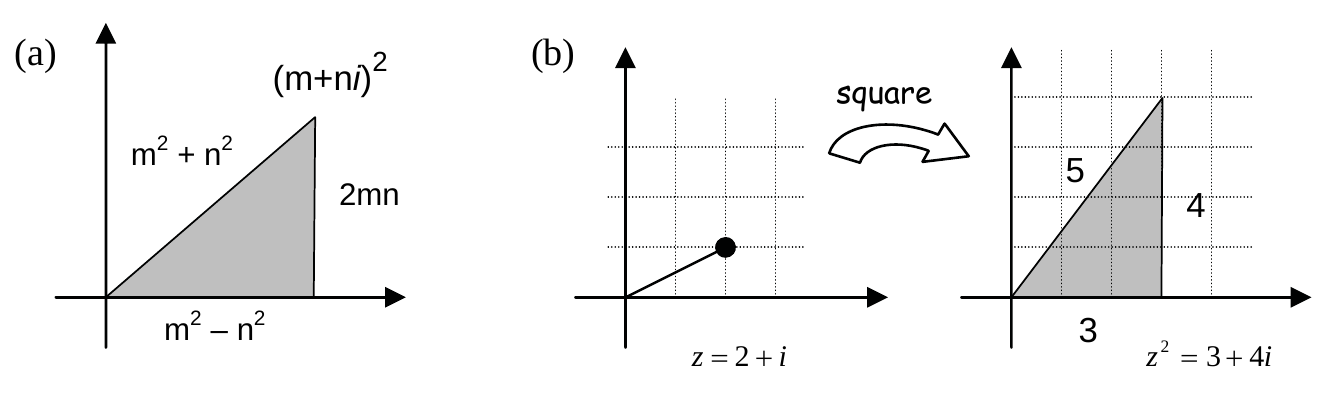}
\]
\caption{(a) Euclid's parameterization.  (b) The Egyptian triangle is the square of $z = 2+i$.}
\label{fig:1}
\end{figure}
Equation \ref{eq:2.1} is equivalent to Euclid's formula for the parameterization of 
Pythagorean triples \cite{T-T}. We shall however allow $z$ to be any integer complex number, 
and therefore admit triangles with negative legs (but not hypotenuses). The squaring map
\[
  \hbox{sq: } \mathbb{C} \to  \mathbb{C}: z  \to  z^{2}
\]
(exclude zero) has a certain redundancy --- both $z$ and -- $z$ give the same 
Pythagorean triple. This double degeneracy has the obvious explanation: 
squaring a complex number has a geometric interpretation of doubling the 
angle. Therefore one turn of the parameter vector $z=m+ni$ around the origin 
makes the Pythagorean vector $z^{2}$ go twice around the origin. 
We know the analogue of such a situation from quantum physics. Fermions need 
to be turned in the visual space twice before they return to the original 
quantum state (a single rotation changes the phase of the ``wave function'' 
by 180$^\circ$). Mathematically this corresponds to the group homomorphism of a double cover
\[
  SU(2) \xrightarrow{2:1} SO(3)
\]
which is effectively exploited in theoretical physics \cite{BL}. 
The rotation group $SO(3)$ is usually represented as the group of special 
orthogonal 3$\times 3$ real matrices, and the unitary group $SU(2)$ is typically 
realized as the group of special unitary 2$\times $2 complex matrices. 
The first acts on $\mathbb{R}^{3}$ --- identified with the visual physical space, 
and the latter acts on $\mathbb{C}^{2}$ --- interpreted as the spin representation 
of states. 
One needs to rotate a visual vector twice to achieve a single rotation in the spinor space $\mathbb{C}^{2}$. 
This double degeneracy of rotations can be observed on the quantum level \cite{AS}, \cite{RZB}.
\begin{figure}[h]  
\[
\includegraphics[height=1.2in]{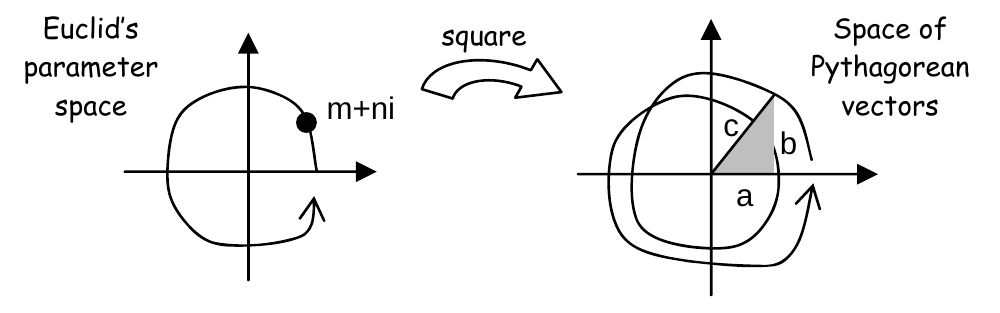}
\]
\caption{Double degeneracy of the Euclid's parameterization.}
\label{fig:2}
\end{figure}
\begin{remark}   
The group of rotations $SO(3)$ has the topology of a 3-dimensional projective space, and as such has first fundamental group isomorphic to $\mathbb{Z}_{2}$. 
Thus it admits loops that cannot be contracted to a point, but a double of a loop (going twice around a topological ``hole'') becomes contractible. 
Based on this fact, P.A.M. Dirac designed a way to visualize the property as the so-called ``belt trick'' \cite{Jur, Kau, New}. 
\end{remark}

This analogy between the relation of the Pythagorean triples to the Euclid's 
parameters and the rotations to spinors has a deeper level --- explored in 
the following sections, in which we shall show that \eqref{eq:2.1} is a shadow of the double covering 
homomorphism $SL^{\pm\!}(2) \to O(2,1)$. 
It shall become clear that calling the Euclid parameterization $(m,n)$ a \textit{Pythagorean spinor} is legitimate. 

Before we go on, note another interesting redundancy of the Euclidean 
scheme. For a complex number $z=m + ni$ define 
\[
  z^{d} = (m+n) + (m-n)i\,.
\]
The square of this number brings
\begin{equation}    \label{eq:2.3}
  (z^{d})^{2}=4mn+2(m^2-n^2)\;i
\end{equation}
with the norm $\vert z^{d}\vert ^{2} = 2(m^{2} + n^{2})$. 
Thus both numbers $z$ and $z^{d}$ give -- up to scale -- the same Pythagorean 
triangle but with the legs interchanged. 
For instance both $z = 2 + i$ and $z^{d} = 3 + i$ give Pythagorean triangles 3-4-5 and 8-6-10, respectively (both similar to the Egyptian triangle). 
In Euclid's recipe \eqref{eq:1.2}, only one of the two triangles is listed. 
This map has ``near-duality'' property: 
\begin{equation}    \label{eq:2.4}
  (z^{d})^{d} = 2z\,.
\end{equation}
%

\section{Pseudo-quaternions and Matrices}  \label{s:3}

\vskip.1in
\noindent
\textbf{A. Pseudo-quaternions.} 
Let us introduce here an algebra very 
similar to the algebra of quaternions.

\begin{defin} Pseudo-quaternions $\mathbb{K}$ (kwaternions) are 
numbers of type 
\begin{equation}    \label{eq:3.1}
  q=a+b \mathbf{i} + c \mathbf{j} + d\mathbf{k}
\end{equation}
where $a, b, c, d \in \mathbb{R}$ and where $\mathbf{i}$, $\mathbf{j}$, 
$\mathbf{k}$ are independent ``imaginary units''. 
Addition in $\mathbb{K}$ is defined the usual way. 
Multiplication is determined by the following rules for the ``imaginary units'':
\begin{equation}    \label{eq:3.2}
\begin{alignedat}{2}
  \mathbf{i}^{2} &= 1   &\qquad    \mathbf{ ij} &= \mathbf{k}  \\
  \mathbf{j}^{2} &= 1   &\qquad    \mathbf{ jk} &= -\mathbf{i}  \\
  \mathbf{k}^{2} &= -1 &\qquad    \mathbf{ ki} &= -\mathbf{j}
\end{alignedat}
\end{equation}
plus the anticommutation rules for any pair of distinct imaginary units: 
$\mathbf{ij} = -\mathbf{ji}$, $\mathbf{jk }= -\mathbf{kj}$, and $\mathbf{ki} = 
-\mathbf{ik}$.  
[The rules are easy to remember: the minus sign appears only 
when $\mathbf{k}$ is involved in the product].
\end{defin}

The pseudo-quaternions form an associative, non-commutative, real 4-dimensional algebra with a unit. 
Define the \textit{conjugation} of $q$ as 
\begin{equation}    \label{eq:3.3}
  \overline q = a - b\mathbf{i} - c\mathbf{j} - d\mathbf{k}
\end{equation}
and the squared norm
\begin{equation}    \label{eq:3.4}
  \vert q\vert ^{2} = \overline qq = a^{2} -  b^{2} - c^{2}+d^{2}
\end{equation}
The norm is not positive definite, yet kwaternions ``almost'' form a 
division algebra. 
Namely the inverse
\begin{equation}    \label{eq:3.5}
  q^{-1} = \overline q / \vert q\vert ^{2}
\end{equation}
is well-defined except for a subset of measure zero --- the ``cone'' 
$a^{2} -  b^{2} - c^{2}+d^{2} = 0$.

Other properties like $\vert ab\vert ^{2}=\vert ba\vert ^{2}$, 
$\vert ab\vert ^{2}=\vert a\vert ^{2}\vert b\vert ^{2}$, 
etc., are easy to prove. 
The set $G_{o} = \{q\in \mathbb{K} \mid \vert q\vert ^{2} = 1\}$ forms a group, and $G_{o}  \cong  SL(2,\mathbb{R}$). 
Topologically, the group (subset of kwaternions) is a 
product of the circle and a plane, $G\cong  S^{1}\times 
\mathbb{R}^{2}$ . 
The bigger set $G = \{q \in \mathbb{K} \mid \vert q \vert ^{2}=\pm 1\}$ is isomorphic to $S^{\pm }L(2,R)$, 
the modular group of matrices with determinant equal $\pm $1.

The regular quaternions describe rotations of $\mathbb{R}^{3}$. 
Pseudo-quaternions describe Lorentz transformations of Minkowski space 
$\mathbb{R}^{2,1}$, $O(2,1)$. 
Indeed, note that the map
\begin{equation}    \label{eq:3.6}
    v \to v' = qvq^{-1}  
\end{equation}
preserves the norm, $\vert qvq^{-1}\vert ^{2}=\vert v\vert^{2}$. 
Thus the idea is quite similar to regular quaternions. Represent 
the vectors of space-time $\mathbb{R}^{2,1}$ by the imaginary part
\begin{equation}    \label{eq:3.7}
  \mathbf{v} = x\mathbf{i} + y\mathbf{j} + t\mathbf{k}
\end{equation}

Its norm squared is that of Minkowski space $\mathbb{R}^{2,1}$, $\vert 
\mathbf{v} \vert ^{2} = - x^{2} - y^{2} + z^{2}$. 
Define a \textit{transformation kwaternion} as an element $q \in  G$. 
Special cases are:
\begin{equation}    \label{eq:3.8}
 \begin{alignedat}{3}
	&\hbox{rotation in $xy$-plane:}  &\quad 
		q &= \cos \phi/2 + \mathbf{k} \sin \phi/2 \\
	&\hbox{hyperbolic rotations } &\quad
		q &= \cosh \phi/2 + \mathbf{i} \sinh \phi/2 &\quad
			&\hbox{(boost in direction of $y$-axis)} \\[-3pt]
	&\hbox to .8in{}(boosts): \\[-12pt]
	&&\quad
		q &= \cosh \phi/2 + \mathbf{j} \sinh \phi/2 &\quad
			&\hbox{(boost in direction of $x$-axis)}. 
 \end{alignedat}
\end{equation}

It is easy to see that the norm of $\mathbf{v}$ is preserved transformation 
\eqref{eq:3.7}, as well as its lack of real part. 

\vskip.1in
\noindent
\textbf{B. Pseudo-quaternions as a Clifford algebra. }
The algebra of pseudo-quaternions is an example of a Clifford algebra. 
Namely, it is the Clifford algebra of a pseudo-Euclidean space $\mathbb{R}^{2,1}$, a 3-dimensional 
Minkowski space with quadratic form
\begin{equation}    \label{eq:3.9}
  g = - x^{2} - y^{2} + z^{2}
\end{equation}
The corresponding scalar product of two vectors will be denoted as $g(v, w) = 
\langle v,w \rangle $. 
Let $\{\mathbf{f}_{1}$, $\mathbf{f}_{2}$, $\mathbf{f}_{3}\}$ be an orthonormal basis of $\mathbb{R}^{2,1}$ such that: 
\begin{equation}    \label{eq:3.10}
\begin{aligned}
  \langle  \mathbf{f}_{1}, \mathbf{f}_{1 }\rangle  &= -1  \\
  \langle  \mathbf{f}_{2}, \mathbf{f}_{2 }\rangle &= -1  \\
  \langle  \mathbf{f}_{3}, \mathbf{f}_{3 }\rangle  &= 1 \\
  \hbox{and } \langle  \mathbf{f}_{i }, \mathbf{f}_{j }\rangle  &= 0 
	\hbox{ for any } i \ne  j\,.
\end{aligned}
\end{equation}

Let us build the Clifford algebra \cite{Por} of this space. We shall assume the standard 
sign convention:
\[
  \mathbf{vw} + \mathbf{wv} = - 2 \langle \mathbf{v, w}_{ }\rangle  
\]
(Clifford products on the left side, pseudo-Euclidean scalar product on the 
right). The basis of the universal Clifford algebra $\mathbb{R}_{2,1}$ of 
the space $\mathbb{R}^{2,1}$ consists of eight elements 
\[
\begin{aligned}
  \mathbb{R}_{2,1} = \hbox{span}\{\, 1, \ \mathbf{f}_{1}, \ \mathbf{f}_{2}, \ 
  \mathbf{f}_{3}, \ \mathbf{f}_{1} \mathbf{f}_{2}, \ \mathbf{f}_{2} 
  \mathbf{f}_{3}, \ \mathbf{f}_{3} \mathbf{f}_{1}, \ \mathbf{f}_{1} 
  \mathbf{f}_{2} \mathbf{f}_{3} \,\}
\end{aligned}
\]
The basic elements satisfy these relations (in the sense of the \textit{algebra} products (cf.\ \eqref{eq:3.2}) :
\begin{equation}    \label{eq:3.11}
\begin{aligned}
  \mathbf{f}_{1}^{2} &= 1  \\
  \mathbf{f}_{2}^{2} &= 1  \\
  \mathbf{f}_{3}^{2} &= -1 \\
  \hbox{and } \mathbf{f}_{i} \mathbf{f}_{j} &= - \mathbf{f}_{j} \mathbf{f}_{i} 
	\hbox{ for any } i\ne j\,. 
\end{aligned}
\end{equation}
Relations between the other elements of the basis of the Clifford algebra 
are induced from these via the associativity of the algebra product.

\vskip.1in
\noindent
\textbf{C. Matrix representation of pseudo-quaternions. }
Note that the following four matrices satisfy the above relations \eqref{eq:3.11}:
\begin{equation}    \label{eq:3.12}
  \sigma _{0}=\left[ {{\begin{array}{*{20}c}
		1   & 0   \\
		0   & 1   \\
	\end{array} }} \right] \quad 
   \sigma _{1}=\left[ {{\begin{array}{*{20}c}
		0   & 1   \\
		1   & 0   \\
	\end{array} }} \right] \quad 
   \sigma _{2}=\left[ {{\begin{array}{*{20}c}
		{-1}   & 0   \\
		{\;0}   & 1   \\
	\end{array} }} \right] \quad 
   \sigma _{3}=\left[ {{\begin{array}{*{20}c}
		{\;0}   & 1   \\
		{-1}   & 0   \\
	\end{array} }} \right]\,.
\end{equation}
Indeed, the squares are, respectively:
\begin{equation}   
  \sigma _{0}^2= \sigma _{0} \qquad \sigma _{1}^{2}
	=\sigma_{0} \qquad \sigma _{2}^{2}=\sigma _{0} \qquad \sigma _{3}^{2} 
	= - \sigma _{0}\,.
\end{equation}
The products are 
\begin{equation}    \label{eq:3.14}
  \sigma _{1} \sigma _{2} = + \sigma _{3} \qquad \sigma _{2} 
	\sigma_{3} = - \sigma _{1} \qquad \sigma _{3} \sigma _{1} 
	= -\sigma_{2}
\end{equation}
and $\sigma _{0}\sigma _{i}=\sigma _{i}$ for each $i$. Thus these 
matrices play a similar role as the Pauli matrices in the case of spin 
description or the regular quaternions. Note however the difference: the 
above construct is over real numbers while the Pauli matrices are over 
complex numbers. 

The other elements are represented by matrices as follows:
\[
  \mathbf{1 = }\sigma _{0} \qquad \mathbf{ f}_{1} \mathbf{f}_{2}
	=\sigma_{3} \qquad \mathbf{ f}_{2} \mathbf{f}_{3} 
	= -\sigma _{1} \qquad \mathbf{ f}_{3} \mathbf{f}_{1} 
	= -\sigma _{2} \qquad \mathbf{f}_{1} \mathbf{f}_{2} \mathbf{f}_{3} 
	= -\sigma _{0}\,.
\]
As in the case of quaternions, the \textit{universal} Clifford algebra for $\mathbb{R}^{2,1}$ 
is 8-dimensional, yet the \textit{minimal} Clifford algebra is of dimension 4 and is isomorphic to the algebra of kwaternions $\mathbb{K}$.

\begin{remark}
In algebra, pseudo-quaternions are also known as split quaternions,
especially in the context of  the Cayley–Dickson construction \cite{Di}.  
Other names include para-quaternions, coquaternions  and antiquaternions. 
\end{remark}

\newpage
\section{Minkowski Space of Triangles and Pythagorean Spinors}  \label{s:4}

Now we shall explore the geometry of Euclid's map for Pythagorean triples. 
As a map from a 2-dimensional symplectic space to a 3-dimensional Minkowski 
space, we have: 
\begin{equation}    \label{eq:4.1}
\begin{aligned}
  \varphi : (\mathbb{R}^2 , \omega )  &\to  (\mathbb{R}^{2,1} , G) \\
  (m,n)  &\to  (m^{2 }- n^{2}, 2\textit{mn, m}^{2}+n^{2})
\end{aligned}
\end{equation}
The map \eqref{eq:2.1} of squaring complex numbers is just its truncated version. 
We shall discuss symmetries of the natural structures of both spaces. 
Obviously, we are mostly interested in the discrete subsets $\mathbb{Z}^{2}$ 
and $\mathbb{Z}^{3}$ of those spaces. 

First, introduce the \textit{space of triangles} $(x,y,z)$ as a real 3-dimensional 
Minkowski space $\mathbb{R}^{2,1}$ with a quadratic form 
\begin{equation}    \label{eq:4.2}
  Q = -x^{2} - y^{2} + z^{2}
\end{equation}
(``space-time'' with the hypotenuse as the ``time''), and with metric given 
by a $3\times 3$ matrix $G = diag (1, 1,-1$). The right triangles are 
represented by ``light-like'' (null) vectors, and the Pythagorean triples by 
the integer null vectors in the light cone. Clearly, not all vectors 
correspond to real triangles, and the legs may assume both positive and 
negative values. We have immediately

\begin{proposition}   
The group of integer orthogonal matrices 
$O(2,1;\mathbb{Z}) \subset  O(2,1;\mathbb{R})$ (Lorentz transformations) 
permutes the set of Pythagorean triangles. 
\end{proposition}

On the other hand we have the two-dimensional space of Euclid's parameters 
$\mathbf{E} \cong  \mathbb{R}^{2}$. Its elements will be called 
\textit{Pythagorean spinors.} Occasionally we shall use the isomorphism $\mathbb{R}^{2} \cong \mathbb{C}$ 
and identify $[m,n]^{T}=m+ni$. 

The space of Euclid's parameters will be equipped with an inner product: 

\begin{definition}  
For two vectors $u = [m,n]^{T}$ and $w = [m',n']^{T}$, the value of the symplectic form $\omega $ is defined as
\begin{equation}    \label{eq:4.3}
  \omega (u, w) = mn' - nm'.
\end{equation}
Conjugation $A^\ast$ of a matrix $A$ representing an endomorphism in $\mathbf{E}$ is the adjugate matrix, namely 
\begin{equation}    \label{eq:4.4}
    \hbox{if } A=\left[ {{\begin{array}{*{20}c}
	a   & b   \\
	c   & d   \\
   \end{array} }} \right] 
  \hbox{ then } A^\ast =\left[ {{\begin{array}{*{20}c}
	d & {-b} \\
	{-c} & a \\
  \end{array} }} \right]
\end{equation}
Conjugation of vectors in $\mathbf{E}$ is a map into the dual space, expressed 
in terms of matrices as
\begin{equation}    \label{eq:4.5}
  \left[ {{\begin{array}{*{20}c}
 m   \\
 n   \\
\end{array} }} \right]^\ast \quad =\quad \left[ {{\begin{array}{*{20}c}
 {-n}   & m   \\
\end{array} }} \right]\,.
\end{equation}
Now, the symplectic product may be performed via matrix multiplication: 
$\omega (u,w)=u^{\ast }w$. The map defined by \eqref{eq:4.5} is the symplectic 
conjugation of the spinor. Also, note that $AA^* = A^*A = \det (A) I$. 
\end{definition}

\begin{remark}   
In the complex representation, the symplectic product is $\omega (u,w)=\tfrac{i}{2}(\bar {u}w-\bar {w}u)$.
\end{remark}

\begin{proposition}   
For any two matrices $A$ and $B$ and vector $u$ in the spinor space we have:
\begin{equation}    \label{eq:4.6}
\begin{alignedat}{2}
	(i)& &\qquad  (AB)^* &= B^*A^*	\\
	(ii)& &\qquad (A)^{**} &= A	\\
	(iii)& &\qquad (Au)^* &= u^*A^*\,.
\end{alignedat}
\end{equation}
\end{proposition}

\begin{proposition}   	\label{prop:4.5} 
The group that preserves the symplectic structure (up to a sign) is 
the modular group $SL^{\pm\!}(2,\mathbb{Z}) \subset  SL^{\pm\!}(2,\mathbb{R})$ understood here as 
the group of $2\times 2$ integer (respectively, real) matrices with determinant equal $\pm 1$. 
\end{proposition}

\begin{proof} 
Preservation of the symplectic form is equivalent to matrix property 
$A^*A = \pm I$. Indeed: 
\[
\omega (Au,Aw) = (Au)^*(Aw) = u(A^*A)w = \det(A) u^*w = \pm \omega (u,w)
\]
for any $u ,w$,  since $A^{\ast }A = \pm \det(A)I$, and by assumption det$(A)=\pm 1$. 
Let us relate the two spaces. Given any endomorphism M of the space of 
triangles $\mathbb{R}^{2,1}$, we shall call an endomorphism $\ttilde {M}$ its 
spinor representation, if 
\begin{equation}    \label{eq:4.7}
  M(\varphi (u)) = \varphi  (\ttilde {M}u )\,.
\end{equation}
Now we shall try to understand the geometry of the spin (Euclid's) 
representation of the Pythagorean triples in terms of the kwaternions 
defined in the previous sections. Recall the matrices representing the 
algebra:
\[
\sigma _0 =\left[ {{\begin{array}{*{20}c}
 1   & 0   \\
 0   & 1   \\
\end{array} }} \right]
\quad
\sigma _1 =\left[ {{\begin{array}{*{20}c}
 {\;0}   & 1   \\
 1   & 0   \\
\end{array} }} \right]
\quad
\sigma _2 =\left[ {{\begin{array}{*{20}c}
 {-1}   & 0   \\
 {\;0}   & 1   \\
\end{array} }} \right]
\quad
\sigma _3 =\left[ {{\begin{array}{*{20}c}
 0   & 1   \\
 {-1}   & 0   \\
\end{array} }} \right]\,.
\]
By analogy to the spinor description of Minkowski space explored in 
theoretical physics, we shall build a ``magic correspondence'' for 
Pythagorean triples and their spinor description. First, we shall map the 
vectors of the space of triangles, $\mathbb{R}^{2,1}$ into the traceless 
$2\times 2$ real matrices, $M^{ 0}_{22 }$\,. The map will be denoted by 
tilde $\sim : \mathbb{R}^{2,1}  \to   M^{ 0}_{22 }$ and defined:
\begin{equation}    \label{eq:4.8}
\begin{aligned}
  \mathbf{v} = (x,y,z)  \longrightarrow
\ttilde {\mathbf v} &=x\cdot \frac{1}{2}\left[ 
{{\begin{array}{*{20}c}
 0 & 1 \\
 1 & 0 \\
\end{array} }} \right] + y \cdot \frac{1}{2}\left[ 
{{\begin{array}{*{20}c}
 {-1} & 0 \\
 0 & 1 \\
\end{array} }} \right] + z\cdot \frac{1}{2}\left[ 
{{\begin{array}{*{20}c}
 0 & 1 \\
 {-1} & 0 \\
\end{array} }} \right]  	\\[6pt]
&=\frac{1}{2}\left[ {{\begin{array}{*{20}c}
 {-y}   & {x+z}   \\
 {x-z}   & y   \\
\end{array} }} \right]\,.
\end{aligned}
\end{equation}
\end{proof}

\begin{proposition}   
The matrix representation of the Minkowski space 
of triangles $\mathbb{R}^{2,1}$ realizes the original Minkowski norm via the 
determinant:
\[
\vert \vert \mathbf{v}\vert \vert  =  4 \det \ttilde{\mathbf v} = -x^{2} - y^{2} + z^{2}\,.
\]
The scalar product may be realized by traces, namely: 
\[
\mathbf{v}\cdot \mathbf{w} = -2\,\hbox{\textnormal{Tr}}\ \ttilde {\mathbf{v}}\,\ttilde 
{\mathbf{w}}
\]
and particular vector coefficients may be read from the matrix by:
\[
v^{i} = -\det \sigma_i \cdot \hbox{\textnormal{Tr}}\ \ttilde {\mathbf{v}}\sigma _{i}
\]
\end{proposition}

The technique of Clifford algebras allows one to represent the orthogonal 
group by the corresponding pin group. Since an orthonormal transformation 
may be composed from orthogonal reflections in hyperplanes, one finds a 
realization of the action of the Lorenz group on the Minkowski space of 
triangles via conjugation by matrices of the spin group 
$SL^{\pm\!}(2,\mathbb{R})$; in particular, for any orthogonal matrix $A \in  O(2,1; 
\mathbb{Z})$, the action $\mathbf{v}' = A\mathbf{v}$ corresponds to 
\begin{equation}    \label{eq:4.9}
  \ttilde {\mathbf v}\,'=\ttilde {A}\,\ttilde {\mathbf v}\,\ttilde {A}^\ast 
\end{equation}
that is, the following diagram commutes:
\[
\begin{CD}
\quad \\[-5pt]
\mathbf{v} @>\widetilde{\hbox to 6pt{}}>> \ttilde {\mathbf v} = \frac{1}{2} \sum v^i \sigma_i \\
	@VVAV                                              @VV{\hbox{conj } \ttilde A}V \\
A\mathbf{v} @>\widetilde{\hbox to 6pt{}}>> \ttilde A \ttilde {\mathbf v} \ttilde A^* \\[5pt]
\end{CD}
\]
And now the reward: since the Pythagorean triples lie on the ``light cone'' 
of the Minkowski space, we may construct them from spinors in a manner 
analogous to the standard geometry of spinors for relativity theory. But 
here we reconstruct Euclid's parameterization of the triples. Recall that 
\[
  \left[ {{\begin{array}{*{20}c}
	x  \\
	y  \\
	z  \\
  \end{array} }} \right]\quad =\quad 
  \left[ {{\begin{array}{*{20}c}
	{m^2-n^2}  \\
	{2mn}  \\
	{m^2+n^2}  \\
  \end{array} }} \right]\,.
\]
Thus we have:

\begin{theorem}   
The spin representation of Pythagorean triples splits into a tensor (Kronecker) product: 
\begin{equation}    \label{eq:4.10}
  \ttilde {v} = \frac{1}{2}\left[ {{\begin{array}{*{20}c}
	{-y}   & {x+z}   \\
	{x-z}   & y   \\
  \end{array} }} \right] = \left[ {{\begin{array}{*{20}c}
	{-mn}   & {m^2}   \\
	{-n^2}   & {mn}   \\
  \end{array} }} \right] = \left[ {{\begin{array}{*{20}c}
	m   \\
	n   \\
  \end{array} }} \right]\otimes \left[ {{\begin{array}{*{20}c}
	{-n}   & m   \\
  \end{array} }} \right]\,.
\end{equation}
\end{theorem}

Hence we obtain yet another aspect of the Euclid's formula, namely a tensor 
version of the \eqref{eq:1.2} and \eqref{eq:4.1}:
\begin{equation}    \label{eq:4.11}
  \widetilde{\varphi (u)} = u\otimes u^\ast \,.
\end{equation}
Note that $[-n, m]$ in \eqref{eq:4.10} is the symplectic conjugation of the spinor $[m, n]^{T}$. Now, due to the above Theorem, the adjoint action splits as 
follows
\begin{equation}    \label{eq:4.12}
  \ttilde {\mathbf v}\,' = \ttilde {A}\,\ttilde {\mathbf v}\,\ttilde {A}^\ast 
	= \ttilde {A}\, \left[ {{\begin{array}{*{20}c}
	m   \\
	n   \\
  \end{array} }} \right]\otimes \left[ {-n\;m} \right]\,\ttilde {A}^\ast 
	=\left( {\ttilde {A}\left[ {{\begin{array}{*{20}c}
	m   \\
	n   \\
  \end{array} }} \right]} \right)\otimes \left( {\ttilde {A}\left[ 
  {{\begin{array}{*{20}c}
	m   \\
	n   \\
  \end{array} }} \right]} \right)^{\ast}\,.
\end{equation}
And the conclusion to the story: The spin representation emerges as ``half" 
of the above representation:
\begin{equation}    \label{eq:4.13}
    \left[ {{\begin{array}{*{20}c}
	m   \\
	n   \\
  \end{array} }} \right] \to  \left[ {{\begin{array}{*{20}c}
	{m'}   \\
	{n'}   \\
  \end{array} }} \right] = \ttilde {A}\,\left[ {{\begin{array}{*{20}c}
	m   \\
	n   \\
  \end{array} }} \right]\,.
\end{equation}

\begin{remark}[on d-duality]
The ``duality'' \eqref{eq:2.4} is also represented in 
spin language, namely, the matrix $D$ of exchange of $x$ with $y$ and the 
corresponding 2-by-2 spin matrix, $\ttilde {D}$, are:
\begin{equation}    \label{eq:4.14}
  D=\left[ {{\begin{array}{*{20}c}
	0 & 1 & 0 \\
	1 & 0 & 0 \\
	0 & 0 & 1 \\
  \end{array} }} \right]\quad ,
	\quad
  \ttilde {D}=\left[ {{\begin{array}{*{20}c}
	1 & 1 \\
	1 & {-1} \\
  \end{array} }} \right]\,.
\end{equation}
Indeed, we have:
\[
  \left[ {{\begin{array}{*{20}c}
	1 & 1 \\
	1 & {-1} \\
  \end{array} }} \right]\,\left[ {{\begin{array}{*{20}c}
	m   \\
	n   \\
  \end{array} }} \right]=\left[ {{\begin{array}{*{20}c}
	{m+n}   \\
	{m-n}   \\
\end{array} }} \right]\,.
\]
The $d$-duality may be expressed in the form of a commuting diagram 
\[
\begin{CD}
\quad \\[5pt]
  \left[ {{\begin{array}{*{20}c}
	m   \\
	n   \\
  \end{array} }} \right] 
	@>\widetilde{\hbox to 6pt{}}>>
  \left[ {{\begin{array}{*{20}c}
	x   \\
	y   \\
	z   \\
  \end{array} }} \right] = 
  \left[ {{\begin{array}{*{20}c}
	{m^2-n^2}   \\
	{2mn}   \\
	{m^2+n^2}   \\
  \end{array} }} \right]				\\[10pt]
	@VV{\ttilde D}V  		@VVDV   \\[10pt]
  \left[ {{\begin{array}{*{20}c}
	{m+n}   \\
	{m-n}   \\
  \end{array} }} \right]
	@>\widetilde{\hbox to 6pt{}}>>
  2\,\left[ {{\begin{array}{*{20}c}
	y   \\
	x   \\
	z   \\
  \end{array} }} \right] = \left[ {{\begin{array}{*{20}c}
	{4mn}   \\
	{2m^2-2n^2}   \\
	{2m^2+2n^2}   \\
  \end{array} }} \right]			\\[5pt]
\end{CD}
\]
\end{remark}

\begin{table}[t]
\small{
\renewcommand{\arraystretch}{1.8}
\begin{tabular}{|l|c|l|}	\hline
\multicolumn{3}{|c|}{\textbf{Magic Correspondence}} \\  \hline
			&Minkowski space $\mathbb{R}^{2,1}$
			&\qquad Traceless $2\times 2$ matrices	\\	\hline\hline
Main object	&$\mathbf{v} = (x,y,z)$
			&$\ttilde v = \sum \mathbf{v}^i \sigma_i = \frac{1}{2}
			 { \left[ \begin{smallmatrix}
			        -y		&x+z \\
				x-z 		&y	\end{smallmatrix} \right] }$ 	\\	
Norm		&$\Vert \mathbf{v} \Vert = -x^2-y^2 + z^2$
			&$\Vert\mathbf{v}\Vert = 4 \det \ttilde {\mathbf v}$ 	\\	
Action		&$\mathbf{v}' = A\mathbf{v}$
			&$\ttilde {\mathbf v}' = \ttilde A \ttilde {\mathbf v} \ttilde A^*$    	\\[-7pt]	
			&  ($A \in O(2,1;\mathbb{Z})$) 
			&  ($\ttilde A \in SL^{\pm\!}(2,\mathbb{Z})$)	\\	
Minkowski scalar product
			&$\mathbf{v \cdot w} = \mathbf{v}^TG\mathbf{w}$
			&$\mathbf{v \cdot w} = -2\,\hbox{Tr}\, \ttilde{\mathbf{v}}
				\ttilde{\mathbf{w}}$			\\	
The $i$th coefficient
			&$v^i = \mathbf{v \cdot e}_i$
			&$v^i = - \det\sigma_i \cdot\hbox{Tr } \ttilde{\mathbf{v}} \sigma_1$ \\  \hline
\end{tabular}
}
\vskip.1in
\caption{Magic correspondence for Pythagorean triples and their spinor description}
\label{t:1}
\end{table}

\section{Hall Matrices and Their Spinor Representation}  \label{s:5}

It is known that all primitive Pythagorean triples can be generated by the 
following three \textit{Hall matrices} \cite{Hal}
\begin{equation}    \label{eq:5.1}
  U=\left[ {{\begin{array}{*{20}c}
 1   & 2   & 2   \\
 2   & 1   & 2   \\
 2   & 2   & 3   \\
\end{array} }} \right] \quad L=\left[ {{\begin{array}{*{20}c}
 1   & {-2}   & 2   \\
 2   & {-1}   & 2   \\
 2   & {-2}   & 3   \\
\end{array} }} \right] \quad R=\left[ {{\begin{array}{*{20}c}
 {-1}   & 2   & 2   \\
 {-2}   & 1   & 2   \\
 {-2}   & 2   & 3   \\
\end{array} }} \right]
\end{equation}
by acting on the initial vector $\mathbf{v} = [3,4,5]^{T}$ (``Egyptian 
vector''). For clarification, a few examples:
\[
\begin{alignedat}{2}
  Lv = \left[ {{\begin{array}{*{20}c}
	1   & {-2}   & 2   \\
	2   & {-1}   & 2   \\
	2   & {-2}   & 3   \\
  \end{array} }} \right]\left[ {{\begin{array}{*{20}c}
	3   \\
	4   \\
	5   \\
  \end{array} }} \right] &= \left[ {{\begin{array}{*{20}c}
	5   \\
	{12}   \\
	{13}   \\
  \end{array} }} \right] 
	&\quad
  R\mathbf{v} &= \left[ {{\begin{array}{*{20}c}
	{-1}   & 2   & 2   \\
	{-2}   & 1   & 2   \\
	{-2}   & 2   & 3   \\
  \end{array} }} \right]\;\left[ {{\begin{array}{*{20}c}
	3   \\
	4   \\
	5   \\
  \end{array} }} \right] = \left[ {{\begin{array}{*{20}c}
	{15}   \\
	8   \\
	{17}   \\
  \end{array} }} \right] 					\\[10pt]
  U\mathbf{v} = \left[ {{\begin{array}{*{20}c}
	1   & 2   & 2   \\
	2   & 1   & 2   \\
	2   & 2   & 3   \\
  \end{array} }} \right]\left[ {{\begin{array}{*{20}c}
  3   \\
  4   \\
  5   \\
  \end{array} }} \right]  &= \left[ {{\begin{array}{*{20}c}
	{21}   \\
	{20}   \\
	{29}   \\
  \end{array} }} \right] 
	&\quad
  &\hbox{URL}^{2}U\mathbf{v} = \left[ {{\begin{array}{*{20}c}
	{3115}   \\
	{3348}   \\
	{4573}   \\
  \end{array} }} \right] \,.
\end{alignedat}
\]
Let us state it formally:

\begin{theorem}[Hall]  
The set of primitive Pythagorean triples is in 
one-to-one correspondence with the algebra of words over alphabet $\{U, L, 
R\}$.
\end{theorem}

The original argument of Hall was algebraic (see \cite{Hal} for a proof). But 
here we shall reinterpret this intriguing result in terms of geometry. Hall 
matrices may be understood in the context of our previous section and 
augmented with the spinor description. Let us start with this:

\begin{proposition}   
The Hall matrices and their products are elements of the Lorentz group $O(2,1; \mathbb{Z})$. 
In particular, for any $\mathbf{v} = [x,y,z]^{T}$, 
\[
g(\mathbf{v},\mathbf{v}) =0 \Rightarrow  g(M\mathbf{v}, M\mathbf{v}) = 0,
\]
and therefore they permute Pythagorean triples.
\end{proposition}

\begin{proof} Elementary. Recall that the matrix of the pseudo-Euclidean 
metric is G = diag(1,1,-1). One readily checks that Hall matrices preserve 
the quadratic form, i.e., that $X^{T}GX = G$ for $X=U$, $L$, $R$. In particular we have $L,R  \in  SO(2,1)$, as det $L = 1$, det $R = 1$. Since det $U = -1$, $U$ contains a reflection.
\end{proof}

Hall's theorem thus says that the set of primitive Pythagorean triples 
coincides with the orbit through $[3,4,5]^{T}$ of the action of the 
semigroup generated by the Hall matrices, a subset of the Lorenz group of 
the Minkowski space of triangles:
\[
\hbox{gen} \{R,L,U \} \subset  O(2,1; \mathbb{Z})\,.
\]

Thus the results of the previous section apply in particular to the Hall 
semigroup. In particular:

\begin{theorem}   
The spin representation of the Hall matrices are 
\begin{equation}    \label{eq:5.2}
  \ttilde {U}=\left[ {{\begin{array}{*{20}c}
	2   & 1   \\
	1   & 0   \\
  \end{array} }} \right]
	\quad
  \ttilde {L}=\left[ {{\begin{array}{*{20}c}
	2   & {-1}   \\
	1   & {\;0}   \\
  \end{array} }} \right]
	\quad
  \ttilde {R}=\left[ {{\begin{array}{*{20}c}
	1   & 2   \\
	0   & 1   \\
  \end{array} }} \right]\,.
\end{equation}
In particular, they are members of the modular 
group $SL^{\pm\!}(2,\mathbb{Z}) \subset  SL^{\pm\!}(2,\mathbb{R})$, that is they 
preserve the symplectic form up to a sign.
\end{theorem}

\begin{proof} 
The transformations in the space of Euclid's parameters that 
correspond to the Hall matrices may be easily found with elementary algebra. 
By Proposition~\ref{prop:4.5}, we need to check $M^{\ast }M=\pm I$. Simple 
calculations show:
\[
\ttilde {U}^\ast \ttilde {U}=-I
	\qquad
\ttilde {L}\,^\ast \ttilde {L}=I
	\qquad
\ttilde {R}^\ast \ttilde {R}=I\,.
\]
That is, the determinants are $\det\ttilde {U}= -1$, $\det \ttilde {L} = 1$, 
and $\det \ttilde {R} = 1$.
\end{proof}

Clearly, the ``magic correspondence'' outlined in the previous section holds 
as well for the Hall matrices. In particular, the spin version of matrices 
act directly on Pythagorean spinors \eqref{eq:4.9}, and the action on the Pythagorean 
vectors may be obtained by the tensor product \eqref{eq:4.13}. 

A classification of the semigroups in the modular 
group $SL^{\pm\!}(2,\mathbb{Z})$ that generate the set of primitive Pythagorean spinors as 
their orbits seems an interesting question. 

\vskip.1in
\noindent
\textbf{Remark on the structure of the Pythagorean semigroup and its spin 
version.}
The last unresolved question concerns the origin or the geometric 
meaning of the Hall matrices and their spin version. First, one may try to 
interpret Hall matrices in terms of the Minkowski space-time structure. They 
may easily be split into a boost, spatial rotation and reflection. Indeed, 
define in $\mathbb{R}^{2,1}$ these three operators: 
\[
H=\left[ {{\begin{array}{*{20}c}
 3 & 0 & {\sqrt 8 } \\
 0 & 1 & 0 \\
 {\sqrt 8 } & 0 & 3 \\
\end{array} }} \right] \,,
\quad
T=\left[ {{\begin{array}{*{20}c}
 c & s & 0 \\
 {-s} & c & 0 \\
 0 & 0 & 1 \\
\end{array} }} \right] \quad \hbox{ where  } c = s = \surd 2/2
\]
($H = $ boost by ``velocity" ($\surd $3/8,0) , $T =$ rotation by $45^\circ$). 
Then the Hall matrices are the following Lorentz transformations :
\[
\begin{aligned}
U &= T^{2 }(T^{-1}HT) = THT  \\
R &= THT R_{1}  \\
L &= THT R_{2}
\end{aligned}
\]
where $R_{1} = \hbox{diag}(-1,1,1)$ and $R_{2} = \hbox{diag}(1,-1,1)$ represent 
reflections. For instance, $ U$ represents a boost in the special direction (1,1) 
followed by space point-inversion. This path seems to lead to nowhere. Thus 
we may try the spinor version of the Hall matrices. One may easily see that 
the latter can be expressed as linear combinations of our pseudo-Pauli 
basis:
\[
\ttilde {L}\ = \sigma _{0} - \sigma _{2} - \sigma _{3} 
	\qquad
\ttilde {R} =\sigma _{0}+\sigma _{1}+\sigma _{3} 
	\qquad 
\ttilde  {U} =\sigma _{0}+\sigma _{1} - \sigma _{2}\,.
\]
But this na\"{\i}ve association does not seem to explain anything. Quite 
surprisingly, insight may be found in the geometry of disk packing. 
This is the subject of the next section.

\section{Pythagorean Triangles, Apollonian Gasket and Poincar\'{e} Disk}  \label{s:6}

Now we shift our attention to --- at first sight rather exotic for our 
problem --- Apollonian gaskets. 

\vskip.1in
\noindent
\textbf{Apollonian window.} 
Apollonian gasket is the result of the following construction. 
Start with a unit circle, called in the following the \textit{boundary circle}.
Inscribe two circles so that all three are mutually tangent. 
Then inscribe a new circle in every enclosed triangular-shaped region (see Figure 3a). 
Continue \textit{ad infinitum}. 
A special case, when the first two inscribed circles are half the radius of 
the boundary circle will be called the \textit{Apollonian window} (shown in Figure~\ref{f:3}b). 
\begin{figure}[ht]
\[
  \includegraphics[width=3.7in]{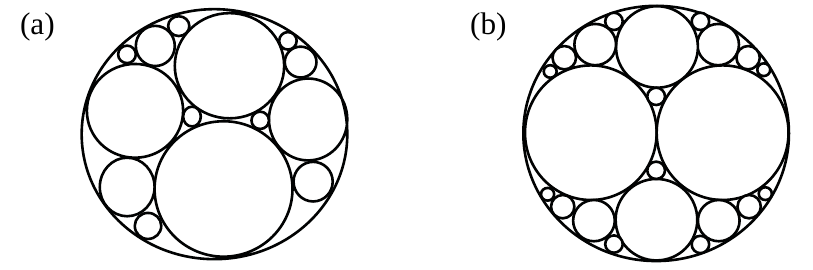}
\]
\caption{An Appollonian gasket and the Apollonian window}
\label{f:3}
\end{figure}
The Apollonian window has amazing geometric properties \cite{Mum}, \cite{Man}. 
One of them is the fact that the curvature of each circle is an integer. 
Also, as demonstrated in \cite{LMW}, the center of each circle has rational coordinates. 
It can be shown that the segment that joins the centers of any two adjacent 
circles forms the hypotenuse of a Pythagorean triangle, whose two legs are 
parallel/perpendicular to the main axes. 
More specifically, the sides of the triangle become integers when divided by the 
product $r_{1}r_{2}$ of the radii of the adjacent circles (see Figure~\ref{f:4}). 
\begin{figure}[h]
\[
  \includegraphics[width=4.4in]{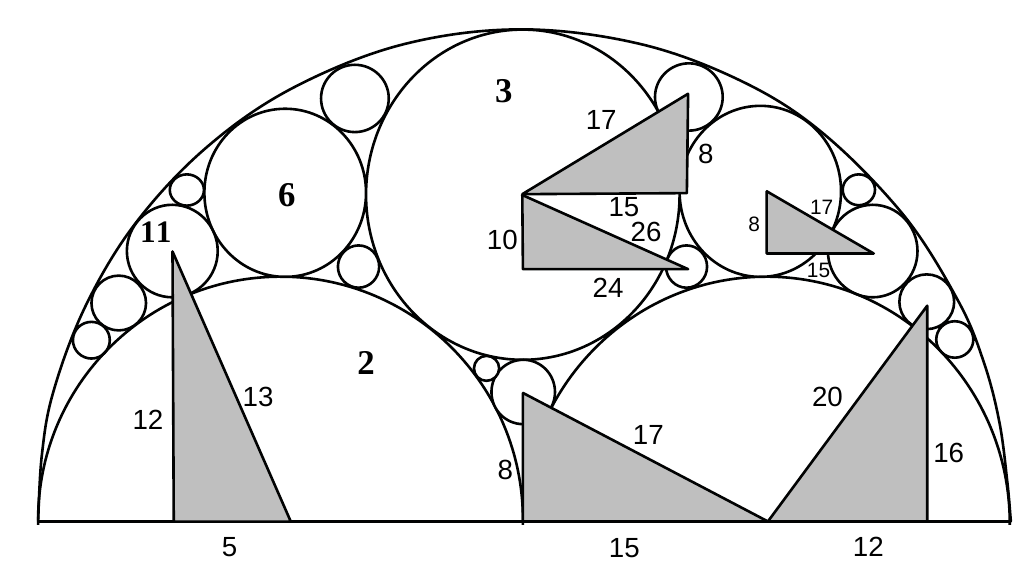}
\]
\caption{Pythagorean triangles in the Apollonian window. Bold numbers represent the curvatures of the corresponding circles.}
\label{f:4}
\end{figure}

Let us call a \textit{subboundary circle} any circle in the Apollonian window tangent to the boundary 
circle (shaded circles in Fig.~\ref{f:5}). Consider a pair of tangent circles of 
which one is the boundary--- and the other a subboundary circle. If we 
prolong the hypotenuses of the associated triangles, we shall hit points on 
the boundary circle, namely the points of tangency. Due to this 
construction, the slope of each such line is rational. We will try to see 
how to permute these points. 

Among the many symmetries of the Apollonian window are inversions in the 
circles that go through the tangency points of any three mutually tangent 
circles. Such inversions permute the disks of the window, and in particular 
preserve their tangencies. We shall look at the following three symmetries 
labeled A, B, C (see Figure~\ref{f:5}b): 
\begin{equation}    \label{eq:6.1}
\begin{aligned}
  A &\hbox{ -- reflection through the vertical axis} \\ 
  B &\hbox{ -- reflection through the horizontal axis} \\
  C &\hbox{ -- inversion through the circle $C$ (only a quarter of the circle is shown)} 
\end{aligned}
\end{equation}

\begin{proposition}   
The three compositions of maps
\begin{equation}    \label{eq:6.2}
    CA, \quad CB, \quad CBA
\end{equation}
leave the set of subboundary circles of the first quadrant invariant. In 
particular, they permute points of tangency on the boundary circle in the 
first quadrant.
\end{proposition}

\begin{proof} A reflection in line $A$ or $B$ or their composition $AB$ 
(reflection through the central point) carries any circle in the first 
quarter to one of the other three quarters. If you follow it with the 
inversion through $C$, the circle will return to the first quarter. Since 
tangency is preserved in these transformations, the proposition holds. 
\end{proof}
\begin{figure}
\[
  \includegraphics[width=4.4in]{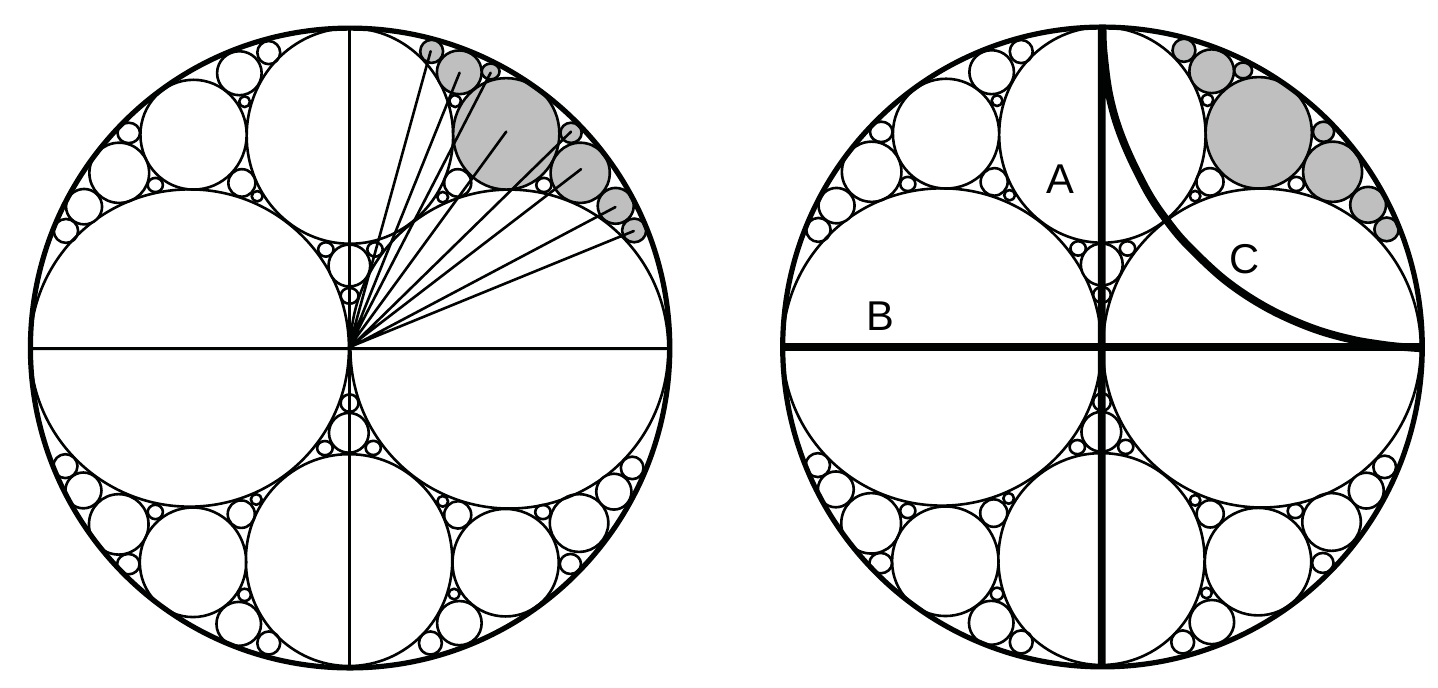}
\]
\caption{(a) Subboundary circles and the corresponding Pythagorean rays.  
(b) Three symmetries of the Apollonian window.}
\label{f:5}
\end{figure}
The crucial observation is that both the lines $A$, $B$ and circle $C$ may be 
understood as ``lines'' in the Poincar\'{e} geometry, if the circle is 
viewed as the Poincar\'{e} disk. This will allow us to represent these 
operations by matrices using the well-known hyperbolic representation of the 
Poincar\'{e} disk.

\vskip.1in
\noindent
\textbf{Poincar\'{e} disk. }Recall the geometry of the hyperbolic 
Poincar\'{e} disk. In the standard model, the set of points is that of a 
unit disk $D$ in the Euclidean plane:
\[
  D = {\{}(x,y) \in  \mathbb{R}^{2 }: x^{2} + y^{2} 
	< 1{\}} \hbox{ with } \partial D = {\{}(x,y) : x^{2} + y^{2} = 1{\}}.
\]
Poincar\'{e} \textit{lines} are the circles that are orthogonal to $\partial D$  
(clearly, only the intersection with $D$ counts). This geometry -- as is well known -- 
may be induced from a hyperbolic linear space. Consider a three-dimensional Minkowski 
space $\mathbb{R}^{2,1}$ and a hyperboloid $H$:
\begin{equation}    \label{eq:6.3}
    t^{2} - x^{2} - y^{2} = 1\,.
\end{equation}
Stereographic projection $\pi $ onto the plane $P\in  \mathbb{R}^{2,1}$ 
defined by $t = 0$, with the vertex of projection at $(-1,0,0)$, brings all 
points of the hyperboloid $H$ onto $D$ in a one-to-one manner. In particular, 
each Poincar\'{e} line in $D$ is an image of the intersection of a plane in 
$\mathbb{R}^{2,1}$ through the origin $\mathbb{O}$ with the hyperboloid $H$, 
projected by $\pi $ onto $D$. We shall use this plane-line correspondence. 
Recall also that reflection in a plane $P$ can be done with the use of a unit 
normal vector $\mathbf{n}$:
\begin{equation}    \label{eq:6.4}
    R_{n}: \mathbf{v}\quad \to\quad   \mathbf{v}' 
	= \mathbf{v} - 2\frac{\langle \mathbf{v, n}\rangle }{\langle \mathbf{n, n}\rangle }    \mathbf{n}
\end{equation}
where the orthogonality $\mathbf{n}\bot P$ and the scalar product are in the 
sense of the pseudo-Euclidean structure of the Minkowski space 
$\mathbb{R}^{2,1}$.

\vskip.1in
\noindent
\textbf{Back to the Apollonian window. }Consider the three symmetries of the 
Apollonian window \eqref{eq:6.2}. Each of them may be realized in terms of a 
reflection in a corresponding plane in the hyperbolic representation. 

\begin{proposition}   
The three symmetries \eqref{eq:6.1} of the Poincar\'{e} 
disc (coinciding with the Apollonian window) have the following matrix 
representations 
\begin{equation}    \label{eq:6.5}
\begin{aligned}
  \hbox{Symmetry }A: \qquad &\mathbf{n}_{1} = [1, 0, 0]^{T} 
	\quad \to  \quad R_{1}
	=\left[   {{\begin{array}{*{20}c}
	{-1}   & 0   & 0   \\
	0   & 1   & 0   \\
	0   & 0   & 1   \\
  \end{array} }} \right] 				\\
  \hbox{Symmetry }B: \qquad &\mathbf{n}_{2} = [0, 1, 0]^{ T} 
	\quad \to  \quad R_{2}
	=\left[ {{\begin{array}{*{20}c}
	1   & 0   & 0   \\
	0   & {-1}   & 0   \\
	0   & 0   & 1   \\
  \end{array} }} \right]				\\
  \hbox{Symmetry }C: \qquad &\mathbf{n}_{3} = [1, 1, 1]^{ T} 
	\quad \to  \quad R_{3}
	=\left[ {{\begin{array}{*{20}c}
	{-1}   & {-2}   & 2   \\
	{-2}   & {-1}   & 2   \\
	{-2}   & {-2}   & 3   \\
  \end{array} }} \right]\,.
\end{aligned}
\end{equation}
\end{proposition}

\begin{proof} One can easily verify that each $\mathbf{n}_{i}$ is unit and 
corresponds to the assigned symmetry. Here are direct calculations for 
finding the matrix corresponding to the third symmetry $C$ corresponding to
$\mathbf n = \mathbf n_3$.    Acting on the basis 
vectors and using \eqref{eq:6.4} we get:
\[
\begin{aligned}
  R_{3} \mathbf{e}_{1} 
	= \mathbf{e}_{1}    - 2 \frac{\langle \mathbf{e_1, n}\rangle }{\langle \mathbf{n, n}\rangle }  
        &=\left[ {{\begin{array}{*{20}c}
		1   \\	0   \\	0   \\
	\end{array} }} \right]    -   2\cdot    {\small\frac{-1}{-1}}
          \left[ {{\begin{array}{*{20}c}
		1   \\	1   \\	1   \\
	\end{array} }} \right]=\left[ {{\begin{array}{*{20}c}
		{-1}   \\	{-2}   \\	{-2}   \\
	\end{array} }} \right]					\\ \\
  R_{3} \mathbf{e}_{2} 
	= \mathbf{e}_{2}   - 2 \frac{\langle \mathbf{e_2, n}\rangle }{\langle \mathbf{n, n}\rangle } 
	&=\left[ {{\begin{array}{*{20}c}
		0   \\	 1   \\	 0   \\
	\end{array} }} \right] - 2\cdot    {\small\frac{-1}{-1}}
          \left[ {{\begin{array}{*{20}c}
		1   \\	 1   \\	 1   \\
	\end{array} }} \right]=\left[ {{\begin{array}{*{20}c}
		{-2}   \\	 {-1}   \\	 {-2}   \\
	\end{array} }} \right] 					\\ \\
  R_{3} \mathbf{e}_{3} 
	= \mathbf{e}_{3}    - 2 \frac{\langle \mathbf{e_1, n}\rangle }{\langle \mathbf{n, n}\rangle } 
	&=\left[ {{\begin{array}{*{20}c}
	0   \\	 0   \\	 1   \\
	\end{array} }} \right] - 2\cdot    {\small\frac{1}{-1}}
          \left[ {{\begin{array}{*{20}c}
	1   \\	 1   \\	 1   \\
	\end{array} }} \right]=\left[ {{\begin{array}{*{20}c}
	2   \\	 2   \\	 3   \\
\end{array} }} \right] 
\end{aligned}
\]
which indeed defines matrix $R_{3}$. The other two, $R_{1}$ and $R_{2}$ are 
self-explanatory.
\end{proof}

And now we have our surprising result: 

\begin{theorem}   
The hyperbolic representation of the permutations \eqref{eq:6.2} of the subboundary 
circles of the first quadrant correspond to the Hall matrices: 
\begin{equation}    \label{eq:6.6}
\begin{aligned}
  R_{3} R_{1} &=\left[ {{\begin{array}{*{20}c}
	1   & {-2}   & 2   \\
	2   & {-1}   & 2   \\
	2   & {-2}   & 3   \\
  \end{array} }} \right] = L		\\
  R_{3} R_{2} &=\left[ {{\begin{array}{*{20}c}
	{-1}   & 2   & 2   \\
	{-2}   & 1   & 2   \\
	{-2}   & 2   & 3   \\
  \end{array} }} \right] = R		\\
  R_{3}R_{1}R_{2} &=  \  \left[ {{\begin{array}{*{20}c}
	1   & 2   & 2   \\
	2   & 1   & 2   \\
	2   & 2   & 3   \\
  \end{array} }} \right] \ = U\,.
\end{aligned}
\end{equation}
\end{theorem}

\begin{corollary} Each primitive Pythagorean triangle is represented in 
the Apollonian window. 
\end{corollary}

This finally answers our question of Section~\ref{s:5} on the origin and structure 
of the matrices of spin representation of the Hall matrices. Following the 
above definitions, we get in Clifford algebra the following representations 
of the reflections that constitute the symmetries \eqref{eq:6.2}:
\begin{equation}    \label{eq:6.7}
\begin{aligned}
  &(\mathbf{f}_{1} + \mathbf{f}_{2} + \mathbf{f}_{3} ) \mathbf{f}_{1} 
	= \mathbf{f}_{1}^{2} + \mathbf{f}_{2}\mathbf{f}_{1} + \mathbf{f}_{3} 
	\mathbf{f}_{1} = -\sigma _{3}+\sigma _{0} - \sigma _{2} = 
	\left[ {{\begin{array}{*{20}c}
		2   & {-1}   \\
		1   & {\;0}   \\
	\end{array} }} \right]=\ttilde {L}			\\
  &(\mathbf{f}_{1} + \mathbf{f}_{2} + \mathbf{f}_{3} ) \mathbf{f}_{2} 
	= \mathbf{f}_{1}\mathbf{f}_{2} + \mathbf{f}_{2}^{2} + \mathbf{f}_{3} 
	\mathbf{f}_{2}=\sigma _{0}+\sigma _{3}+\sigma _{1} = 
	\left[ {{\begin{array}{*{20}c}
		1   & 2   \\
		0   & 1   \\
	\end{array} }} \right]=\ttilde {R}			\\
  &(\mathbf{f}_{1} + \mathbf{f}_{2} + \mathbf{f}_{3} ) (-\mathbf{f}_{3} ) 
	= -\mathbf{f}_{1}\mathbf{f}_{3} - \mathbf{f}_{2}\mathbf{f}_{3} - \mathbf{f}_{3}^{2} 
	= -\sigma _{2}+\sigma _{1}+\sigma _{0}=\left[ 
	{{\begin{array}{*{20}c}
		2   & 1   \\
		1   & 0   \\
	\end{array} }} \right]=\ttilde {U}
\end{aligned}
\end{equation}
(cf.\ \eqref{eq:5.2}). Note that the duality map \eqref{eq:4.13} that exchanges $x$ and $y$ in the 
Pythagorean triangles and which has a spinor representation $D = \sigma _{1} - \sigma _{2}$ fits the picture, 
too, as the vector $\mathbf{n} = [1,-1,\, 0]$ represents (scaled) reflection in the plane containing the $x$-$y$ diagonal.

We conclude our excursion into the spin structure of the Euclid parameterization of 
Pythagorean triples with Figure~\ref{f:6} that shows spinors corresponding to some 
Pythagorean triangles in the Apollonian window.

\begin{figure}
\[
  \includegraphics[width=4.2in]{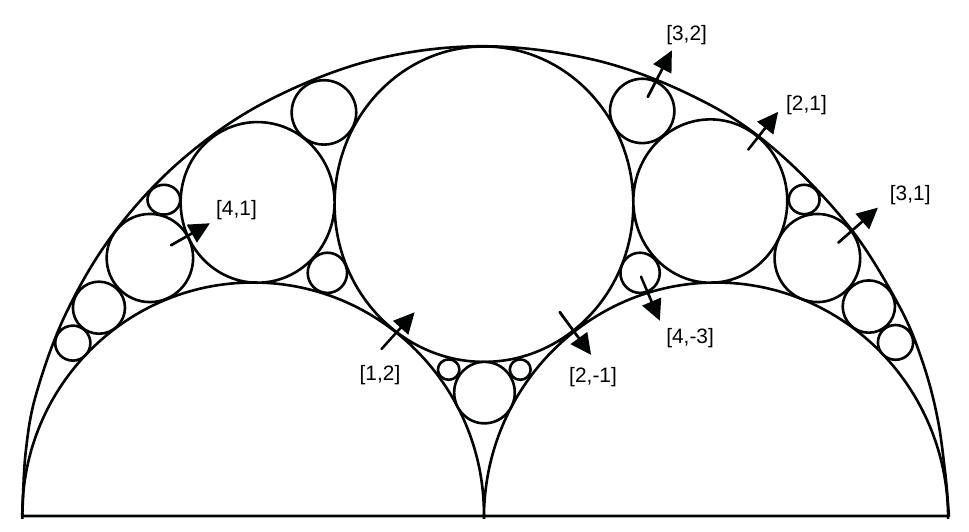}
\]
\caption{Spinors for some Pythagorean triangles in the Apollonian window, 
given as labels at the points of tangency of the corresponding pairs of circles.}
\label{f:6}
\end{figure}

\section{Conclusions and Remarks}  \label{s:7}

We have seen that the space of Euclid's parameters for Pythagorean triples 
is endowed with a natural symplectic structure and should be viewed as a 
spinor space for the Clifford algebra $\mathbb{R}_{2,1 }$, built over 
3-dimensional Minkowski space, whose integer light-like vectors represent 
the Pythagorean triples (see Figure~\ref{f:7}). The minimal algebra for 
$\mathbb{R}^{2,1}$ is four-dimensional and may be conceptualized as 
``pseudo-quaternions'' $\mathbb{K}$ and represented by $2\times 2$ matrices. 
In this context the Pythagorean triples may be represented as traceless 
matrices, and Euclid's parameterization map as a tensor product of spinors. 
This set-up allows us to build the spinor version of the Hall matrices. The 
Hall matrices acquire a geometric interpretation in a rather exotic context 
of the geometry of the Apollonian window. 

\begin{figure}
\[
  \includegraphics[width=5.5in]{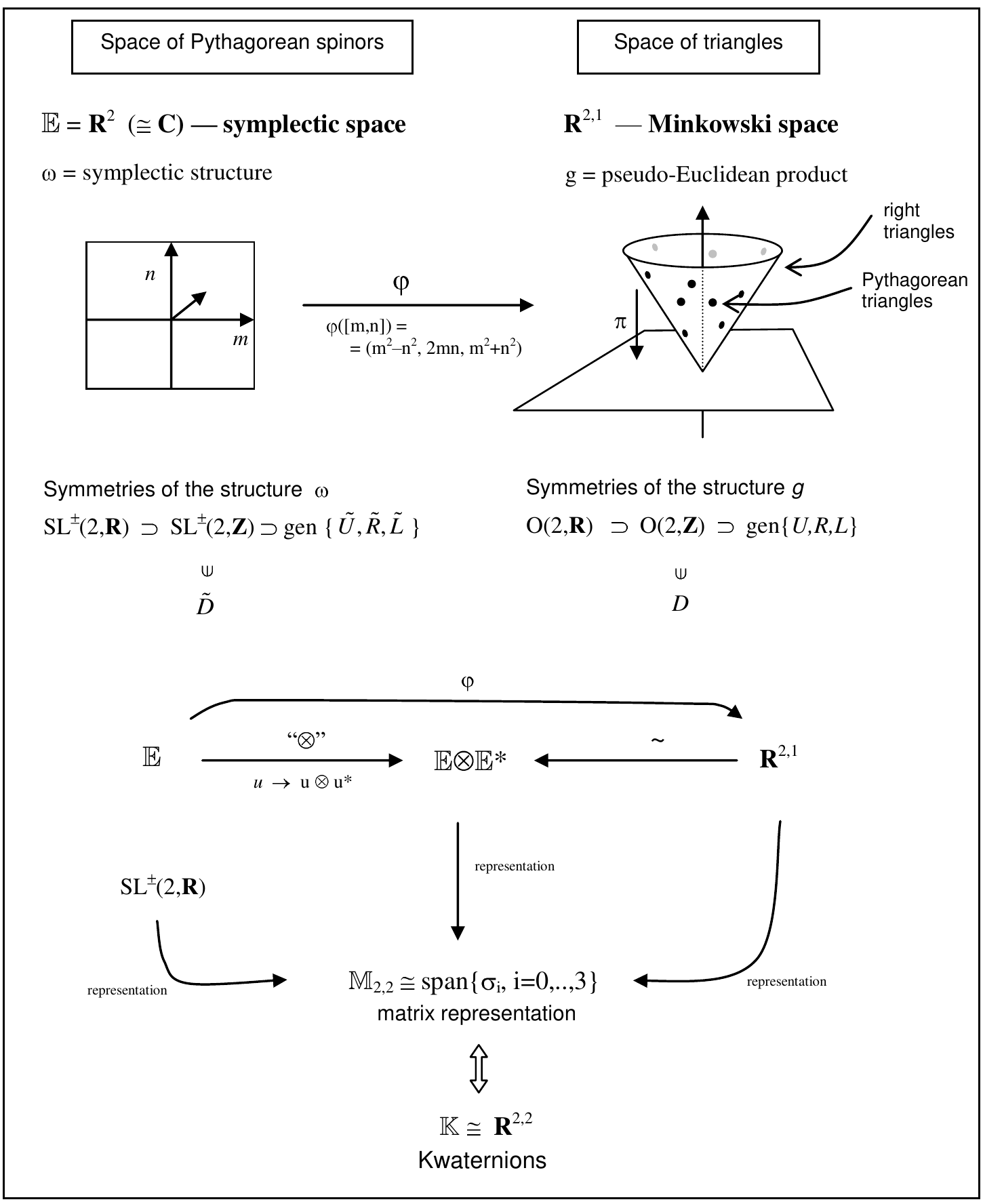}
\]
\caption{Objects related to Euclid's parametrization of Pythagorean triples.}
\label{f:7}
\end{figure}

Euclid's discovery of the parameterization of Pythagorean triples may be 
viewed then as the first recorded use of a spinor space. The spinor 
structure of the Apollonian window is another interesting subject that will 
be studied further elsewhere.

The method may be generalized to other dimensions, as indicated here: 

\vskip.1in
\noindent
\textbf{Method A.} Define a Pythagorean $(k,l)$-tuple as a system of $(k+l)$ 
integers that satisfy
\begin{equation}    \label{eq:7.1}
  a_{1}^{2}+a_{2}^{2} + a_{3}^{2} + {\ldots} + a_{k}^{2} = 
b_{1}^{2} + b_{2}^{2} + b_{3}^{2} + {\ldots} + b_{l}^{2}
\end{equation}
In order to obtain a parameterization of Pythagorean $(k,l)$-tuples do the 
following:
%
Start with pseudo-Euclidean space $\mathbb{R}^{k,l}$, build a representation 
of the Clifford algebra $\mathbb{R}_{k,l}$ . Then split the matrix that 
represents the isotropic vectors of $\mathbb{R}^{k,l}$  
into a tensor product of spinors. This tensor square provides the 
parameterization when restricted to the integer spinors. 

\begin{example} 
Consider Pythagorean quadruples, that is quadruples of 
integers $(a, b, c, d)$
\[
a^{2}+b^{2}+c^{2}=d^{2}\,.
\]
Among the examples are (1,2,2,3), (1,4,8,9), (6,6,7,11), etc. A well-known 
formula that produces Pythagorean quadruples is \cite{Mor}: 
\begin{equation}    \label{eq:7.2}
\begin{aligned}
  a &= 2\textit{mp}			\\
  b &= 2\textit{np}				\\
  c &=p^{2} - (m^{2 }+n^{2})		\\
  d &=p^{2}+m^{2}+n^{2}\,.
\end{aligned}
\end{equation}
It is also known that not all quadruples are generated this way, for 
instance (3, 36, 8, 37) is excluded \cite{Wei}. 
\end{example}

Let us try our method. The spinor representation of the Clifford algebra 
$\mathbb{R}_{3,1}$ is well known. The light-like vectors are represented by 
the Hermitian $2\times 2$ matrices that split into a spinor product,
\begin{equation}    \label{eq:7.3}
  M =\left[ {{\begin{array}{*{20}c}
	{d+a}   & {b+ci}   \\
	{b-ci}   & {d-a}   \\
  \end{array} }} \right] = 2 \left[ {{\begin{array}{*{20}c}
	z   \\
	w   \\
  \end{array} }} \right]\otimes \left[ {{\begin{array}{*{20}c}
	{\bar {z}}   & {\bar {w}}   \\
  \end{array} }} \right],
\end{equation}
where the matrix $M$ on the left side is well-known in physics with $d$ standing 
for time and $a, b, c$ for spatial variables. Note that indeed 
$\hbox{det}\,M=d^{2}-a^{2}-b^{2}-c^{2}$ represents the quadratic form of Minkowski space. 
The factor of 2 is chosen to keep things integer. 

Let $z=m + ni$ and $w=p + qi$. Then \eqref{eq:7.3} becomes: 
\[
  \left[ {{\begin{array}{*{20}c}
	{d+a}   & {b+ci}   \\
	{b-ci}   & {d-a}   \\
  \end{array} }} \right] = 2
  \left[ {{\begin{array}{*{20}c}
	{m^2+n^2}   & {(mp+nq)+(np-mq)i}   \\
	{(mp+nq)-(np-mq)i}   & {p^2+q^2}   \\
  \end{array} }} \right]
\]
which may be readily resolved for $a,b,c,d$. Thus we have proven:

\begin{theorem}   
The following formulae produce all Pythagorean quadruples:
\begin{equation}    \label{eq:7.4}
  \begin{aligned}
  a &=m^{2}+n^{2} - p^{2} - q^{2}		\\
  b &= 2(\textit{mp} + \textit{nq})	\\
  c &= 2(\textit{np} - \textit{mq})	\\
  d &=m^{2}+n^{2}+p^{2}+q^{2}
\end{aligned}
\end{equation}
that is
\[
(m^{2} + n^{2} - p^{2} - q^{2 })^{2} + (2mp + 2nq)^{2} + (2nn 
- 2mq)^{2} = (m^{2} + n^{2} + p^{2} + q^{2 })^{2}\,.
\]
\end{theorem}

The quadruple (3, 36, 8, 37) that was not covered by \eqref{eq:7.2} may be now obtained 
by choosing $(m,n,p,q) = (4,2,4,1)$ 
\[
(4^{2} + 2^{2} - 4^{2} - 1^{2 })^{2} + (2\cdot 4\cdot 4 + 
2\cdot 2\cdot 1)^{2} + (2\cdot 2\cdot 4 - 2\cdot 4\cdot 
1)^{2} = (4^{2} + 2^{2} + 4^{2} + 1^{2 })^{2}\,.
\]

The standard Euclid's parameterization of Pythagorean triples results by choosing 
$n=q=0$ (which imposes $c = 0$), or by $m=n$ and $p =q$ , although this time with a redundant doubling 
in the formulae. Also, choosing only $q = 0$ will result in the system of formulae \eqref{eq:7.2}. 
\\

Quite similarly, we can treat Pythagorean hexads using the fact that 
$\mathbb{R}_{5,1}\cong \mathbb{H}(2)$ (quaternionic $2\times 2$ matrices).

\begin{theorem}   
Let us use collective notation  $m=(m_0, m_1, m_2, m_3)$  and  $n=(n_0, n_1, n_2, n_3)$. 
Also denote in ususal way $mn = m_0n_0 + m_1 n_1 + m_2n_2 + m_3 n_3$.
The following formulae produce Pythagorean hexads:
\begin{equation}    \label{eq:7.5}
  \begin{aligned}
  a_0 &= m^{2}+n^{2} 		\\
  a_1 &= 2(n_0 m_1 - n_1m_0 +  m_3n_2 - m_2 n_3) \\
  a_2 &= 2(n_0 m_2 - n_2m_0 +  m_1n_3 - m_3 n_1) \\
  a_3 &= 2(n_0 m_3 - n_3m_0 +  m_2n_1 - m_1 n_2) \\
  a_4 &=  2 mn  \\
  a_5 &=  m^{2} - n^{2}
\end{aligned}
\end{equation}
that is for any $m, n\in \mathbb{Z}^4$ we have 
\[
    a_0^2 = a_1^2 + a_2^2 + a_3^2 + a_4^2 +a_5^2\,.
\]
\end{theorem}

\begin{proof}
A general Hermitian quaternionic matrix can be split into the Kronecker product
\begin{equation}    \label{eq:7.6}
  \left[ {{\begin{array}{*{20}c}
	a_0+a_5   & a   \\
	\bar {a}  & a_0-a_5   \\
  \end{array} }} \right] = 2 \left[ {{\begin{array}{*{20}c}
	p   \\
	q   \\
  \end{array} }} \right]\otimes \left[ {{\begin{array}{*{20}c}
	{\bar {p}}   & {\bar {q}}   \\
  \end{array} }} \right],
\end{equation}
where on the left side  $a_0, a_5 \in \mathbb{R}$, and  $a=a_1+a_2i+a_2j+a_3k \in \mathbb{H}$, 
and on the right side $q = m_0 + m_1i + m_2j + m_3k$  and  $p = n_0 + n_1i + n_2j + n_3k$.  
Resolving this equation gives (\ref{eq:7.5}).    
\end{proof}

For example $(m,n) = ((1,2,2,1),(2,1,1,1))$  produces  $2^2 + 3^2 + 4^2 + 8^2 + 14^2 = 17^2$.
Note that in this parameterization map, $\varphi:\mathbb{Z}^8\to\mathbb{Z}^6$,  the dimension 
of the parameter space outgrows that of the space of Pythagorean tuples.
\\

Reconsidering the entries of matrices of equation (\ref{eq:7.6}) gives 
us this alternative generalization:
\begin{theorem}    \label{th:big}
Consider equation
\begin{equation}    \label{eq:7.7}
  \left[ {{\begin{array}{*{20}c}
	a+b   & c   \\
	\bar {c}  & a-b   \\
  \end{array} }} \right] = 2 \left[ {{\begin{array}{*{20}c}
	p   \\
	q   \\
  \end{array} }} \right]\otimes \left[ {{\begin{array}{*{20}c}
	{\bar {p}}   & {\bar {q}}   \\
  \end{array} }} \right],
\end{equation}
with $a, b \in \mathbb{R}$, and  $c,p,q \in \mathbb{A}$, 
where $\mathbb{A}$ is an algebra with a not necessarily positive definite norm and cojugation
satisfying $aa^*=|a|^2$ and $(ab)^*=b^*a^*$. If the quadratic form 
of $\mathbb{A}$ is of signature $(r,s)$, 
then formula (\ref{eq:7.5}) produces Pythagorean $(r+1, s+1)$-tuples.    
\end{theorem}
Note that as a special case we may use the Cliford algebras themselves as algebra $\mathbb{A}$.
As an example consider ``duplex numbers" $\mathbb{D}$, \cite{Koc}, which form Clifford algebra 
of ${\mathbb{R}^1}$.  Let $c=c_0+c_1I\in\mathbb{D}$, where $c_0,c_1\in\mathbb{R}$,
and $I^2=1$ (pseudo-imaginary unit).  Similarly, set $p=p_0+p_1I$ and $q=q_0+q_1I$.
Then (\ref{eq:7.7}) produces Pythagorean (2,2)-tuples:
\[
            (p_0^2-p_1^2+q_0^2-q_1^2)^2 
           +(2p_0q_1-2p_1q_0)^2 =
            (p_0^2-p_1^2-q_0^2+q_1^2)^2 
           +(2p_0q_0-2p_1q_1)^2 
\]
A simple application that goes beyond Clifford algebras: 
using the algebra of octonions, $\mathbb{A}=\mathbb{O}$,
results in a parameterization
$\varphi:\mathbb{Z}^{16}\to\mathbb{Z}^{10}$ of Pythagorean ``decuples" by 16 parameters.

\begin{remark}
Since Clifford algebras are $2^n$-dimensional, using them in Theorem
\ref{th:big} will lead to parameteriztion of (generalized)
Pythagorean ($2^{n}+2$)-tuples by   $2^{n+1}$ ``Euclid's parameters".
For $n=0,1,2,\ldots$ we get 3-, 4-, 6-, 10-, ..., -tuples.  
Incidentally, these numbers occur frequently in various string theories,
the reason for which is not fully understood \cite{Sch}.
\end{remark}

\section*{Acknowledgements}

The author thanks Philip Feinsilver for his 
encouragement and helpful remarks. 
He owes also a debt of gratitude to Carl Riehm;
a number of inconsistencies could be fixed thanks to his careful reading of the paper.

\end{document}